\documentclass{amsart}
\usepackage{amsmath, amssymb,graphicx,mathrsfs,enumerate}
\usepackage{amsfonts}
\usepackage{hyperref}
\usepackage{pgf,tikz}

\theoremstyle{plain}
\newtheorem{thm}{Theorem}
\newtheorem{thmA}{Theorem}

\newtheorem{propA}[thmA]{Proposition}

\newtheorem{prop}[thm]{Proposition}
\newtheorem{lem}[thm]{Lemma}

\newtheorem{cor}[thm]{Corollary}

\newtheorem{question}[thmA]{Question}

\newcommand{\calU}{{\mathcal U}}

\newcommand{\calD}{{\mathcal D}}
\newcommand{\calJ}{{\mathcal J}}
\newcommand{\calI}{{\mathcal I}}

\newcommand{\ZZ}{{\mathbb Z}}
\newcommand{\NN}{{\mathbb N}}

\newcommand{\cn}{\operatorname{cn}}
\newcommand{\ecn}{\operatorname{ecn}}
\newcommand{\SL}{\operatorname{SL}}

\newcommand{\rk}{\operatorname{rk}}

\newcommand{\supp}{\operatorname{supp}}

\newcommand{\OR}{\operatorname{OR}}

\makeatletter
\@namedef{subjclassname@2020}{%
  \textup{2020} Mathematics Subject Classification}
\makeatother

\begin{document}

\title[Products of conjugacy classes in simple algebraic groups]{Products of conjugacy classes in simple algebraic groups in terms of diagrams}
\author{Iulian I. Simion}
\address[Iulian I. Simion]
        { Department of Mathematics\\
          Babeș-Bolyai University\\
          Str. Ploieşti 23-25, Cluj-Napoca 400157, Romania\\
          and
          Department of Mathematics\\
          Technical University of Cluj-Napoca\\
          Str. G. Bari\c tiu 25, Cluj-Napoca 400027, Romania}
\email{iulian.simion@ubbcluj.ro}
\thanks{I am grateful to Prof. Attila Mar\'oti for many discussions on this topic. This work was supported by a grant of the Ministry of Research, Innovation and Digitalization, CNCS/CCCDI–UEFISCDI, project number PN-III-P1-1.1-TE-2019-0136, within PNCDI III}

\maketitle
\begin{abstract}
  For a simple algebraic group $G$ over an algebraically closed field 
  we study products of normal subsets.
  For this we mark the nodes of the Dynkin diagram of $G$. 
  We use two types of labels,
  a binary marking and a labeling with non-negative integers.
  The first is used to recognize large conjugacy classes which appear in a product of two conjugacy classes while the second is used to keep track of multiplicities of regular diagrams.
  In particular, we formulate sufficient conditions in terms of marked diagrams, for a product of normal subsets in $G$ to contain regular semisimple elements.
\end{abstract}

\section{Introduction}

A normal subset of a group $G$ is defined to be a non-empty union of conjugacy classes in $G$. By a result of Liebeck and Shalev \cite[Theorem 1.1]{Liebeck_Shalev}, there is a universal constant $c$ such that whenever $N$ is a non-central normal subset in a non-abelian finite simple group $G$ then $N^{k} = G$ for any integer $k$ at least $c \cdot (\log_{2} |G|/ \log_{2} |N|)$. This was generalized by Mar\'oti and Pyber in \cite[Theorem 1.2]{Maroti_Pyber_diameter_bound}, where they prove that there exists a universal constant $c$ such that if $N_{1}, \ldots , N_{k}$ are non-central normal subsets in a non-abelian finite simple group $G$ satisfying $\prod_{i=1}^{m} |N_{i}| \geq |G|^{c}$, then $N_{1} \cdots N_{k} = G$.

The \emph{covering number $\cn(G,N)$} of a group $G$ by a normal subset $N$ of $G$ is the smallest integer $k$ such that $N^k=G$ or $\infty$ if no such $k$ exists. The \emph{covering number $\cn(G)$} of $G$ is the smallest integer $k$ such that $N^k=G$ for any normal subset $N$ of $G$ which is not contained in any proper normal subgroup of $G$. The \emph{extended covering number $\ecn(G)$} of $G$ is the smallest integer $k$ such that $N_1\cdots N_k=G$ whenever $N_1,\ldots, N_k$ are normal subsets of $G$ not contained in any proper normal subgroup of $G$ \cite{Gordeev_Saxl_1}.

By results of Gordeev \cite{Gordeev_1_2}, if $G$ is a simple algebraic group defined over an algebraically closed field of characteristic $0$ then $\cn(G)\leq 4\cdot\rk(G)$ where $\rk(G)$ is the Lie rank of $G$. This result was extended by Ellers, Gordeev and Herzog \cite{Ellers_Gordeev_Herzog} to the case of quasisimple Chevalley groups. More precisely, they show that for such a group $G$ we have $\cn(G)\leq 2^{13}\cdot \rk(G)$. Gordeev and Saxl \cite{Gordeev_Saxl_1} show that there is a constant $c$ such that for any Chevalley group $G$ defined over any field, we have $\ecn(G)\leq c\cdot \rk(G)$. Moreover, if $G$ is a Chevalley group defined over an algebraically closed field, they show that $\ecn(G)\leq 4\cdot \rk(G)$. Extending \cite[Theorem 1.1]{Liebeck_Shalev}, upper bounds on covering numbers of unipotent conjugacy classes are given in terms of their dimension and in terms of their (co)ranks in \cite{2021_unipotent} for $G$ a simple algebraic group over an algebraically closed field $k$ of good characteristic. Recall that the characteristic $p$ of $k$ is good for $G$ if $p\neq 2$ when $G$ is not of type $A$, $p\neq 3$ if $G$ is an exceptional group and $p\neq 5$ if $G$ is of type $E_8$. Theorem C in \cite{2021_unipotent} is extended and improved by \cite[Theorem 1]{Liebeck_Simion} where it is shown that there exists
an absolute constant $c\leq 120$ such that whenever $C$ is a non-central
conjugacy class of a simple algebraic group $G$ then $C^k = G$ for any
integer $k$ at least $c \cdot (\dim(G)/ \dim(C))$.

The methods used in the context of Chevalley groups give explicit constants for the upper bounds on the (extended) covering numbers which are missing in \cite[Theorem 1.1]{Liebeck_Shalev} and \cite[Theorem 1.2]{Maroti_Pyber_diameter_bound}. On the other hand the results in \cite{Liebeck_Shalev,Maroti_Pyber_diameter_bound} take into account the size of the normal subsets and should have analogous statements for Chevalley groups. Bridging the two types of results should entail an analysis of simple algebraic groups over algebraically closed fields which are easier to deal with than finite simple groups of Lie type or Chevalley groups over arbitrary fields, yet resembles these type of groups closely through the BN-pair structure. Central in this context is to measure the size of the normal subset $C_1C_2$ for two conjugacy classes $C_1$ and $C_2$ of $G$.

The product of small conjugacy classes in simple groups grows rapidly in the following sense. By a result of Liebeck, Schul and Shalev \cite[Theorem 1.3]{Liebeck_Schul_Shalev}, given any $\epsilon > 0$, there exists $\delta > 0$ such that if $N_1$, $N_2$ are normal subsets of a non-abelian finite simple group $G$ satisfying $|N_{i}| \leq {|G|}^{\delta}$ for $i = 1$, $2$, then $|N_{1}N_{2}| \geq (|N_{1}| |N_{2}|)^{1-\epsilon}$. An analogue of this statement for algebraic groups is \cite[Theorem 1.5]{Liebeck_Schul_Shalev}. Given any $\epsilon > 0$, there exists $\delta > 0$ such that if $C_1$ and $C_2$ are conjugacy classes in a simple algebraic group $G$ defined over an algebraically closed field and satisfying $\dim (C_{i}) \leq \delta \dim (G)$ for $i = 1$, $2$, then the product $C_{1}C_{2}$ contains a conjugacy class of dimension at least $(1 - \epsilon)(\dim (C_1) + \dim (C_2))$. 

Products of large conjugacy classes cover $G$ rapidly in the following sense. By a result of Gow \cite[Theorem 2]{Gow}, if $G$ is a finite simple group of Lie type, then for any two regular semisimple conjugacy classes $C_1$ and $C_2$ the product $C_1C_2$ contains any non-identity semisimple element of $G$. Hence, the product of four such classes equals $G$. Many of the results on large classes are motivated by Thomposon's conjecture. The analogue for a simple algebraic group $G$ states that there exists a conjugacy class $C$ such that $C^2=G$. We refer to the survay in \cite{Malle_Ore} for more background on this. 
For a simple algebraic group $G$ over an algebraically closed field the product of $4$ regular conjugacy classes equals $G$ (see for example \cite[Lemma 2.1]{Liebeck_Simion}).
When considering a product $N_1\cdots N_k$ of several normal subsets which equals $G$ one would like to understand which normal subsets can make up such a product. By the above, if $k=4$ then regular classes are possible. As $k$ increases, one would like to understand which smaller classes can be used in such a product.

From a different perspective, the Arad-Herzog conjecture states that the product of two non-trivial conjugacy classes in a non-abelian finite simple group $G$ is never a conjugacy class in $G$. Guralnick, Malle and Tiep \cite{Guralnick_Malle_Tiep} prove a strong version of the Arad-Herzog conjecture for simple algebraic groups and in particular show that almost always the product of two conjugacy classes in a simple algebraic group consists of infinitely many conjugacy classes. Guralnick and Malle \cite{Guralnick_Malle} classify pairs of conjugacy classes in almost simple algebraic groups whose product consists of finitely many classes.

\bigskip

In this paper, $G$ denotes a simple algebraic group defined over an algebraically closed field of characteristic $p > 0$. To a normal subset $N\subseteq G$ we attach a set of marked diagrams by means of the representatives in a Borel subgroup $B$. More precisely, for $g\in N\cap B$, the marked diagram of $g$ is the Dynkin diagram of $G$ in which we mark the node corresponding to a simple root $\alpha$ if the projection of $g$ on the root group $U_\alpha$ is not $1$ (see \S\ref{marked_diagrams}). The set of all these diagrams - obtained from elements in $N\cap B$ - is denoted by $\calD(N)$. If $p$ is a good prime for $G$, i.e. $p\neq 2$ if $G$ is not of type $A$, $p\neq 3$ if $G$ is an exceptional group and $p\neq 5$ if $G$ is of type $E_8$, then marked diagrams extend the notion of distinguished diagrams used in the Bala-Carter classification of unipotent conjugacy classes \cite[\S5.11]{Carter_Finite}.

Marked diagrams offer a way of measuring the `size of a conjugacy class' not only for algebraic groups over algebraically closed fields but for simple groups of Lie type and Chevalley groups over arbitrary fields as well: for a small conjugacy class $C$ the set $\calD(C)$ contains diagrams with few marked nodes while for a large class $C$ there are diagrams with many marked nodes in $\calD(C)$. If $\calD^{\circ}$ denotes the diagram with all nodes marked then $\calD^{\circ}\in\calD(N)$ whenever $N$ contains a regular conjugay class (see Propositions \ref{regular_fully_marked} and \ref{u_reg}). Moreover, for two classes $C_1$ and $C_2$ of $G$ if $D_1\in \calD(C_1)$ and $D_2\in \calD(C_2)$ then $D_1\boxplus D_2\subseteq \calD(C_1C_2)$ where $D_1\boxplus D_2$ is the marked diagram obtained by marking exactly those nodes which are marked in both $D_1$ and $D_2$ (see Proposition \ref{sum_diagram_Nor}).

Marked diagrams can also be viewed as elements $\sum_{\alpha\in\Delta} n_{\alpha} \alpha$ of the monoid $\NN\Delta=\NN^{|\Delta|}$ where $\Delta$ is a set of simple roots and $n_{\alpha}\in\NN$ (see \S\ref{monoids}). In this notation $\calD^{\circ}$ is the marked diagram $\sum_{\alpha\in\Delta}\alpha$. Let $D = \sum_{\alpha\in\Delta} n_{\alpha} \alpha$ and $D' = \sum_{\alpha\in\Delta} m_{\alpha} \alpha$ be two diagrams with $n_{\alpha},m_{\alpha}\in\NN$. Their sum is $D+D'=\sum_{\alpha\in\Delta}(n_{\alpha}+m_{\alpha})\alpha$. There is a partial order $\geq$ on $\NN\Delta$ defined by $D \geq D'$ if and only if $n_{\alpha} - m_{\alpha} \geq 0$ for all $\alpha \in \Delta$. This gives a way of addressing questions on products of classes in terms of calculations in the monoid $\NN\Delta$. A first statement in this direction is the following.

\begin{propA}
  \label{normal_sets_prod_regular_G}
  Let $G$ be a simple algebraic group, defined over an algebraically closed field. 
  Let $N_1,\dots,N_k$ be normal subsets of $G$ and let $D_{i} \in \calD(N_{i})$ for all $1 \leq i \leq k$.
  If $\sum_{i=1}^{k} D_{i}\geq 12\cdot \rk(G)\cdot\calD^{\circ}$, then $\prod_{i=1}^{k}N_i=G$.
\end{propA}

Comparing this result to \cite{Gordeev_Saxl_1}, where $\ecn(G)$ is shown to be less than $4\cdot \rk(G)$, we don't obtain anything new in terms of the extended covering number of $G$. The condition on the diagrams $D_i$ implies that there are at least $12\cdot\rk(G)$ conjugacy classes in the product, which by \cite{Gordeev_Saxl_1} has to equal $G$.

Since the product of two open subsets of $G$ equals $G$, it is natural to ask when a product $N_1\cdots N_k$ of normal subsets contains an open subset of $G$. For this, notice that any set $A$ of diagrams is partially ordered. For a diagram $D\in A$, let $A(\geq D):=\{E\in A:E\geq D\}$. We say that $A$ is of type $(r,s)$ with respect to the diagrams $D_1,\ldots,D_r\in A$ if there is a partition of $A$ into $r$ subsets $A_1,\dots, A_r$ such that $|A_i\cap A(\geq D_i)|\geq s$ for all $1\leq i\leq r$.

\begin{propA}
  \label{normal_sets_prod_order}
  Let $G$ be a simple algebraic group, defined over an algebraically closed field. 
  Let $N_1,\dots,N_k$ be normal subsets of $G$ and let $D_{i} \in \calD(N_{i})$ for all $1 \leq i \leq k$. If the set $\{D_i:1\leq i\leq k\}$ is of type $(r,6)$ with respect to $D_{i_1},\dots, D_{i_r}$ and $\sum_{j=1}^{r} D_{i_j}\geq \calD^{\circ}$ then $\dim\prod_{i=1}^{m}N_i=\dim G$.
\end{propA}

Since the product of four subsets of $G$, each of which contains a regular conjugacy class, equals $G$, it is natural to ask when a product of normal subsets $N_1\cdots N_n$ contains a regular element.

\begin{thmA}
  \label{normal_sets_prod_regular}
  Let $G$ be a simple algebraic group, defined over an algebraically closed field. 
  Let $N_1,\dots,N_k$ be normal subsets of $G$ and let $D_{i} \in \calD(N_{i})$ for all $1 \leq i \leq k$.
  If $\sum_{i=1}^{k} D_{i}\geq 16\calD^{\circ}$, then $\prod_{i=1}^{k}N_i$ contains regular semisimple elements.
\end{thmA}

Proposition \ref{normal_sets_prod_order} shows in particular that if $\sum_{i=1}^{k} D_{i}\geq 4\calD^{\circ}$ and each diagram $D_i$ appears at least $6$ times then the product of the corresponding normal subsets contains an open subset of $G$. It follows that if we replace $6$ by $12$ the product of the corresponding subsets is $G$. Theorem \ref{normal_sets_prod_regular} shows that if the diagrams can be partitioned into $4$ subsets, each of which sum up to a regular diagram (a diagram which is greater than or equal to $\calD^{\circ}$) then the product of the corresponding subsets is $G$. This suggests that it should be possible to remove $\rk(G)$ in Proposition \ref{normal_sets_prod_regular_G}.

\begin{question}
  Let $G$ be a simple algebraic group, defined over an algebraically closed field. 
  Let $N_1,\dots,N_k$ be normal subsets of $G$ and let $D_{i} \in \calD(N_{i})$ for all $1 \leq i \leq k$.
  If $\sum_{i=1}^{k} D_{i}\geq c\cdot\calD^{\circ}$ for some constant which does not depend on $G$, does it follow that $\prod_{i=1}^{k}N_i=G$? 
  \end{question}

An affirmative answer to this question would give 
in particular a means of recognizing which conjugacy classes can appear in a product $N_1\cdots N_k=G$ for a fixed $k$. By Proposition \ref{diags_jordan_classes0}, conjugacy classes in the same Jordan class have the same marked diagrams and Proposition \ref{diags_jordan_classes} exhibits conjecturally maximal marked diagrams of a conjugacy class. Notice that it suffices to give an answer to the above question for classical groups of heigh rank, a treatment of the bounded rank case is needed for a good upper bound on the constant $c$.

\bigskip
  
The paper is structured as follows: in \S\ref{notation_background} we fix notation, we collect results on conjugacy classes in algebraic groups which are relevant to marked diagrams and we recall a factorization of $G$ which is needed in the sequel. In \S\ref{marked_diagrams} we introduce the notion of a marked diagram and the associated monoids. In \S\ref{marked_diagrams_unipotent} we describe the link between marked diagrams and unipotent elements. The proofs of Propositions \ref{normal_sets_prod_regular_G} and \ref{normal_sets_prod_order} and of Theorem \ref{normal_sets_prod_regular} are given in \S\ref{diagrams_and_conjugacy_classes}.

  \section{Preliminaries}
\label{notation_background}

In this paper $G$ denotes a simple algebraic group of rank $r=\rk(G)$ defined over an algebraically closed field $F$ of characteristic $p>0$. 
We fix a Borel subgroup $B$ with unipotent radical $U$ and maximal torus $T$. We let $\Phi$ denote the roots of $G$ with respect to $T$, the set of positive roots $\Phi^{+}$ are with respect to $U$ and $\Delta$ denotes the set of simple roots of $\Phi$ in $\Phi^{+}$. We denote by $U^{-}$ the radical of the Borel subgroup opposite to $B$, i.e. $U=U^{\dot w_0}$ for some representative $\dot w_0\in N_G(T)$ of the longest element (with respect to $\Delta$) of the Weyl group $N_G(T)/T$. 

For each root $\alpha\in\Phi$ let $u_{\alpha}:F\rightarrow U_{\alpha}$ be an isomorphism from the additive group of the ground field $F$ onto the root subgroup $U_{\alpha}$. For each $\alpha\in\Phi$ we denote by $\alpha^{\vee}:F^{\times}\rightarrow T$ the cocharacter corresponding to the root $\alpha$. Then
\begin{equation}
  \label{cochar_action}
        {}^{\alpha^{\vee}(t)}u_{\beta}(x)
        =\alpha^{\vee}(t)u_{\beta}(x)\alpha^{\vee}(t)^{-1}
        =u_{\beta}(\beta(\alpha^{\vee}(t))x)
        =u_{\beta}(t^{\langle\beta,\alpha\rangle}x)
\end{equation}
for all $\alpha,\beta\in\Phi$, $t\in F^{\times}$, $x\in F$ (see \cite[II\S1.3]{Jantzen_Reductive} and \cite[Ch.7]{Carter_Simple}).

Any element $g\in B$ has a unique factorization of the form $g=s\prod_{\alpha\in\Phi^{+}}u_{\alpha}(x_{\alpha})$ for some $s\in T$, $x_{\alpha}\in F$ and where the product is in a fixed (but arbitrary) ordering of $\Phi^{+}$ (see for example \cite[Theorem 11.1]{Malle_Testerman}). The projections $g\mapsto u_{\alpha}(x_\alpha)$ depend in general on the ordering of $\Phi^{+}$. However, the projections on simple root groups do not depend on this order as can be seen from the commutator relations (see for example \cite[Theorem 11.8]{Malle_Testerman}). We point out that $[U,U]\subseteq \prod_{\alpha\in\Phi^{+}-\Delta}U_{\alpha}$ (this can be deduced from \cite[Proposition 11.5]{Malle_Testerman} and the commutator relations).
In what follows we make use of this fact without further notice. For a unipotent element $u=\prod_{\alpha\in\Phi^{+}}u_{\alpha}(x_{\alpha})$ we denote by $\supp(u)$ the set of simple roots $\alpha$ with the property that $x_{\alpha}\neq 0$. For an element $s\in T$ we denote by $\supp(s)$ the set of simple roots $\alpha$ with the property that $\alpha(s)\neq 1$.

For an element $g\in G$ we denote by $g_s$ and $g_u$ the semisimple and the unipotent part in the Jordan decomposition of $g$ respectively: $g=g_sg_u=g_ug_s$. In any algebraic group, all Borel subgroups, respectively all maximal tori are conjugate and any element in $G$ is conjugate to an element in $B$. Thus, conjugating if necessary, we may assume that $g$ lies in $B$ and that $g_s$ lies in $T$: we may conjugate $B$ and $T$ by the same element in $G$, or, when $g$ is a representative of a conjugacy classe, we may replace $g$ by a $G$-conjugate with the above property. 

For a set of roots $I\subseteq\Phi$, let $\Phi_{I}$ be the root subsystem generated by $I$, i.e. $\Phi_{I}=\ZZ I\cap\Phi$. We denote by $L_I$ the subgroup $\langle T,U_{\alpha}:\alpha\in\Phi_{I}\rangle$ of $G$. If the roots in $I$ are simple then $L_I$ is a standard Levi subgroup. In this case, we denote by $P_I$ the standard parabolic subgroup with Levi factor $L_I$. When we need to specify the ambient group $G$, we write $L_I^{G}$ or $P_I^{G}$. Notice that the notation $L_I^{G}$ and $P_I^{G}$ makes sense in the more general case of a reductive algebraic group $G$. In the particular case of $I=\{\alpha\}\subseteq\Delta$ we denote by $P_{\alpha}$ the parabolic subgroup $P_{I}$ and by $G_{\alpha}$ the subgroup generated by $U_{\pm\alpha}$.

\subsection{Semisimple conjugacy classes}
\label{semisimple_classes}
For an element $g\in G$ we have $C_G(g)=C_{C_G(g_s)}(g_u)$. Hence, describing the conjugacy class of $g$ entails two parts: the description of $C_G(g_s)$ and the description of unipotent conjugacy classes in $C_G(g_s)$. The structure of the centralizer of a semisimple element in $G$ is known. In the following theorem we extract a combinatorial description which we use in the description of Jordan classes in Section \ref{jordan_classes}. Recall that the connected components of the centralizers of semisimple elements in $G$ are called pseudo-Levi subgroups \cite{McNinch_Sommers}. The Levi-envelope of a pseudo-Levi $H$ is the minimal Levi subgroup $L$ of $G$ containing $H$ such that $Z(H)^{\circ}=Z(L)^{\circ}$ (see \cite[\S3]{Carnovale_Esposito}).

\begin{thm}[Centralizers of semisimple elements]
  \label{good_rep_ss}
  Let $G$ be a simple algebraic group and let $s$ be a semisimple element contained in the maximal torus $T$. There is a subset $I\subseteq\Delta$ such that exactly one of the following holds:
  \begin{enumerate}
  \item $C_G(s)^{\circ}$ is the Levi-subgroup $L_I$, or
  \item there is a root $\gamma\notin\Delta$ and a simple root $\beta\in\Delta$ such that $C_G(s)^{\circ}$ is the proper pseudo-Levi subgroup $L_{I\cup\{\gamma\}}$ with Levi-envelope $L_{I\cup\{\beta\}}$.
    \end{enumerate}
\end{thm}

\begin{proof}
  The subgroup $M_s=C_G(s)^{\circ}$ is the pseudo-Levi subgroup given by $\langle T,U_{\alpha}:\alpha(g_s)=1\rangle$ \cite[II \S4.1]{Springer_Steinberg}. Let $Z_s$ denote the center of $M_s$. Then $L_s=C_G(Z_s^{\circ})$ is a Levi subgroup, the Levi envelope of $M_s$ (see \cite[\S3]{Carnovale_Esposito}). Conjugating, we may assume that it is a standard Levi subgroup, i.e. $L_s=L_J$ for some $J\subseteq\Delta$.

  The pseudo-Levi $M_s$ is a subgroup of $L_s$ and the torus $Z_s^{\circ}$ is a maximal central torus of $M_s$ and of $L_s$ \cite[Lemma 3.7]{Carnovale_Esposito}. Factoring we obtain the semisimple subgroup $M_s/Z_s^{\circ}$ of the semisimple group $L_s/Z_s^{\circ}$. Under the projection $L_s\rightarrow L_s/Z_s^{\circ}$, $x\mapsto \bar x$, the centralizer of $\bar s$ in $\bar L_s$ is $\bar M_s$.

  Since $Z_s^{\circ}$ is contained in the maximal torus $T$ which lies in $M_s$ and $L_s$, the projection $x\mapsto \bar x$ induces a bijection $\alpha\mapsto \bar \alpha$ on the roots of $L_s$ w.r.t. $T$ and the roots of $\bar L_s$ w.r.t. $\bar T$. The root system $\bar \Phi_J$ decomposes into a direct sum of irreducible root systems $\bar\Phi_i$ with $0\leq i\leq m$ for some integer $m$. Since $G$ is a simple algebraic group and $L_s$ is a standard Levi subgroup of $G$, at most one of the $\bar\Phi_i$ is not of type $A$. Renumbering, we may assume that $\bar\Phi_0$ has this property.

  Conjugating if necessary, we may assume that the root system of $\bar M_s$ with respect to $\bar T$ has a basis $\bar K\subseteq \bar J\cup \{\bar\gamma\}$ where $-\bar\gamma$ is the highest root of $\bar\Phi_0$ \cite[Proposition 30]{McNinch_Sommers}. 

  The rank of $\bar M_s$ is at most that of $\bar L_s$, i.e. $|\bar K|\leq |\bar J|$. We claim that $|\bar K|=|\bar J|$. If this is not the case, then $\bar M_s$ has rank at most $|\bar J|-1$, i.e. a maximal torus of $\bar M_s$ has dimension at most $|\bar J|-1$. However the $|\bar J|$-dimensional torus $\bar T$ lies in $\bar M_s$, which is a contradiction.

  It follows that $\bar K$ is either $\bar J$ or $\{\bar\gamma\}\cup(\bar J-\{\bar\beta\})$ for some $\bar\beta\in\bar J\cap\bar\Phi_0$. If $\bar K=\bar J$ then (1) holds with $I=K$. If $\bar K = \{\bar\gamma\}\cup(\bar J-\{\bar\beta\})$ then (2) holds with $I=J-\{\beta\}$. 
\end{proof}

For subsets $I\subseteq\Delta$ and $I'\subseteq\Phi$ the pair $(I,I')$ is called \emph{of proper pseudo-Levi type} if $[L_{I'},L_{I'}]$ is a proper maximal rank subsystem subgroup of $[L_{I},L_{I}]$. The pair $(I,I')$ is called \emph{of proper Levi type} if $I'=I$ is a Levi subgroup. The pair $(I,I')$ is called \emph{of pseudo-Levi type} if it is of proper pseudo-Levi type or of proper Levi type. Examples can be constructed with the Borel-de Siebenthal algorithm (see for example \cite[\S13.2]{Malle_Testerman}). In particular for type $B_r$ one may choose $I=\Delta$ and $I'=\Delta\setminus \{\alpha_r\}\cup\{-\alpha_0\}$ where $\alpha_0$ is the highest root, in order to obtain $[L_{I'},L_{I'}]$ of type $D_r$. Notice also that for type $A_r$ the semisimple conjugacy classes are of proper Levi type.

\subsection{Unipotent conjugacy classes}
\label{unipotent_classes}
For our purposes, we use the Bala-Carter-Pommerening classification of unipotent conjugacy classes \cite{Bala_Carter,Pommerening} (see also \cite[Theorem 5.9.6 and \S5.11]{Carter_Finite}). For this, we require the characteristic of the ground field to be good for $G$, i.e. $p\neq 2$ if $G$ is not of type $A$, $p\neq 3$ if $G$ is of exceptional type and $p\neq 5$ if $G$ is of type $E_8$. The classification of unipotent classes in the case where the characteristic of the ground field is bad for $G$ was achieved through the contribution of many authors. The state of the art for unipotent conjugacy classes is available in \cite{Liebeck_Seitz}.

For an element $g\in G$, the group $C_G(g_s)^{\circ}$ is connected reductive, hence a central product of simple algebraic groups and a central torus with no non-trivial unipotent elements in the center. Therefore, the conjugacy class of $g_u$ is a product of unipotent conjugacy classes in the simple factors of $C_G(g_s)^{\circ}$. The following theorem is well known and translates directly to the case where $G$ is a connected reductive algebraic group. The statement is implicit in \cite[Theorem 1]{Liebeck_Seitz}.

Recall that a unipotent element is distinguished if $C_G(u)^{\circ}$ is unipotent. For a parabolic subgroup $P=LQ$ with Levi factor $L$ and unipotent radical $Q$ we have $\dim L\geq \dim (Q/[Q,Q])$. The group $P$ is a distinguished parabolic subgroup if $\dim L=\dim (Q/[Q,Q])$ \cite[\S2.5-6]{Liebeck_Seitz}. An element $g$ of a parabolic subgroup $P$ is called a Richardson element of $P$ if the $P$-conjugacy class of $g$ intersects $Q$ in an open subset of $Q$. 

\begin{thm}[Conjugacy classes of unipotent elements]
  \label{good_rep_u}
  Let $G$ be a simple algebraic group defined over an algebraically closed field of good characteristic. There is a bijective correspondence between unipotent conjugacy classes of $G$ and $G$-classes of pairs $(L,P)$, where $L$ is a Levi subgroup of $G$ and $P$ is a distinguished parabolic subgroup of $[L,L]$. The $G$-class of $(L,P)$ corresponds to the $G$-conjugacy class containing a Richardson element of $P$.
\end{thm}

For two subsets of simple roots $K\subseteq J\subseteq\Delta$, the pair $(J,K)$ is called \emph{distinguished} if $P_{K}^{[L_J,L_J]}$ is a distinguished parabolic subgroup of $[L_{J},L_{J}]$.

\subsection{Jordan classes}
\label{jordan_classes}
For algebraic groups, Jordan classes were introduced in \cite{Carnovale_Esposito} inspired by the similar notion for the adjoint action of a group on its Lie algebra (see \cite[\S4]{Carnovale_Esposito}). Conjugacy classes in $G$ can be grouped together as follows. Two conjugacy classes are equivalent \cite[\S4]{Carnovale_Esposito} if they have representative $g$ and $h$ respectively, such that $C_G(h_s)^{\circ}=C_G(g_s)^{\circ}$, $h_s\in g_sZ(C_G(g_s)^{\circ})^{\circ}$ and $h_u$ is conjugate to $g_u$ in $C_G(g_s)^{\circ}$. The unions of elements in the corresponding equivalence classes are called Jordan classes. The set of Jordan classes is denoted by $\calJ$. They partition $G$ into a finite number of irreducible normal subsets and the conjugacy classes in a Jordan class have the same dimension.

In view of the previous sections we assume in this section that the characteristic of $F$ is good for $G$ and consider the set
$$
\calI=\Big\{(I,I',J,K):(I,I')\text{ is of pseudo-Levi type},J\subseteq I',(J,K)\text{ is distinguished}\Big\}.
$$
We define a map
$$
\phi:\calI\rightarrow\calJ
$$
as follows. Fix $(I,I',J,K)\in\calI$. By \cite[Proposition 32]{McNinch_Sommers}, since the characteristic of $F$ is good for $G$, for a given pair $(I,I')$ of pseudo-Levi type, there exists a semisimple element $s$ in $G$ such that $C_G(s)^{\circ}$ is $L_{I'}$. By \cite[Proposition 1]{Richardson} there exists a unipotent element $u\in C_G(s)^{\circ}$ such that $u$ is a Richardson element of $P_{K}^{[L_J,L_J]}$. Define $\phi(I,I',J,K)$ to be the Jordan class of $su$. For any pair $(s,u)$ as above, we say that \emph{$su$ realizes the data $(I,I',J,K)$}.

\begin{cor}
  \label{classification}
  The map $\phi:\calI\rightarrow\calJ$ is well-defined and surjective.
\end{cor}

\begin{proof}
  Fix an element $(I,I',J,K)\in \calI$. If $s_1$ and $s_2$ are semisimple elements such that $C_G(s_1)^{\circ}=L_{I'}=C_G(s_2)^{\circ}$ then, by definition, for any $u_1,u_2\in L_{I'}$, the elements $s_1u_1$ and $s_2u_2$ are in the same Jordan class if and only if $u_1$ and $u_2$ are in the same $L_{I'}$-conjugacy class. This is the case since both $u_1$ and $u_2$ are Richardson elements of $P_{K}^{[L_{J},L_{J}]}$. Hence $\phi$ is well defined.

  Let $g$ be a representative of some Jordan class. We may assume that $g_s$ lies in $T$. By Theorem \ref{good_rep_ss} there is a pair $(I,I')$ of pseudo-Levi type such that $C_G(g_s)^{\circ}=L_{I'}$. By \cite[III \S1.14]{Springer_Steinberg}, since we assume that the characteristic of $F$ is good for $G$ we have $g_u\in L_{I'}$. By Theorem \ref{good_rep_u}, $g_u$ is a Richardson element of some distinguished parabolic subgroup $P_{K}^{[L_{J},L_{J}]}$. Hence $\phi$ is surjective.
  \end{proof}

\subsection{Unipotent factorization of $G$}

The following result is due to Vavilov, Smolensky and Sury \cite{VSS2011unitriangular}. A proof can also be obtained with the method used in \cite{UUUU} for finite groups of Lie type. We refer to \cite{UUUU} and the references therein for more background on this result.

\begin{thm}
  \label{thm_UUUU}
  Let $G$ be a simple algebraic group, $B\subseteq G$ a Borel subgroup with maximal torus $T$ and unipotent radical $U$. We have 
  $U\cdot U^{\dot w_0}\cdot U\cdot U^{\dot w_0}=G$
where $\dot w_0$ is a representative in $N_G(T)$ of the longest element of the Weyl group $N_G(T)/T$.
\end{thm}

\begin{cor}
	\label{coro_UUU}
	 For any simple algebraic group $G$ we have $\calU^{3}=G$ 
	 where $\calU$ denotes the variety of unipotent elements in $G$.
\end{cor}

\begin{proof}
	We have $G=U\cdot U^{\dot w_0}\cdot U\cdot U^{\dot w_0}\subseteq U^{U^{\dot w_0}}\cdot U^{U^{\dot w_0}}\cdot U^{\dot w_0}\subseteq \calU^{3}$ by Theorem \ref{thm_UUUU}.
\end{proof}

\section{Marked diagrams}
\label{marked_diagrams}

The \emph{marked diagram} of an element $g\in B$ is the Dynkin diagram of $G$ where we mark the nodes corresponding to the simple roots $\alpha$ for which the projection of $g$ on the root subgroup $U_{\alpha}$ is not $1$. More precisely, any $g\in B$ has a factorization of the form $s\prod_{\alpha\in\Phi^{+}}u_{\alpha}(x_{\alpha})$ with $s\in T$ and $x_{\alpha}\in F$, and we mark the nodes corresponding to those $\alpha\in\Delta$ for which $x_{\alpha}\neq 0$. We denote this diagram with $\calD(g)$ and write $\calD^{\circ}$ for the marked diagram having all nodes marked. We denote by $\supp(D)$ the set of simple roots corresponding to marked nodes of a diagram $D$. Notice that for $s\in T$ and $u\in U$, we have $\supp(\calD(su))=\supp(u)$. In what follows we mention a few examples of marked diagrams.
  
\par\noindent
\textbf{Unipotent elements.} With the obvious choice of $T$ and $B$, the unipotent element of $\SL_5(F)$
    $$
    \begin{bmatrix}
      1 & 1 & 0  & 0  & 0  \\
      0  & 1 & 1 & 0  & 0 \\
      0 & 0 & 1 & 0  & 0 \\
      0 & 0 & 0 &  1  & 1\\
      0 & 0 & 0  &  0  & 1 \\
    \end{bmatrix}
    \quad\text{has marked diagram}\quad
\begin{tikzpicture}
  \draw (0,0) -- (3,0);
  \fill (0,0) circle (3pt);
  \fill (1,0) circle (3pt);
  \fill[white] (2,0) circle (3pt);
  \draw (2,0) circle (3pt);
  \fill (3,0) circle (3pt);
\end{tikzpicture}.
$$
Let $g$ be a unipotent element $g$ of $G=\SL_n(F)$ in Jordan form, i.e. $g=\bigoplus_i J_i$ where the $J_i$ are Jordan blocks. Again, with the obvious choice of $T$ and $B$, the number of connected components of $\calD(g)$, obtained after removing the non-marked nodes, equals the number of Jordan blocks $J_i$ and each connected component $D_i$ corresponds to a Jordan block $J_i$ such that the number of nodes in $D_i$ is one less than the length of $J_i$.

 If $g\in B$ is a distinguished unipotent element, then $\calD(g)$ is obtained from the labeled Dynkin diagram of $G$ by marking the nodes with label `2'. For a list of these diagrams see \cite[\S5.9]{Carter_Finite}.

\par\noindent
\textbf{Semisimple elements.} If $g\in T$, then $g$ is conjugate to an element $g'\in B$ such that $\calD(g')$ has those simple roots $\alpha$ marked for which $\alpha(g')\neq 1$, i.e. $\supp(\calD(g'))=\supp(g)$ (see Proposition \ref{su_reg}).
  
\par\noindent
\textbf{Regular elements.} A regular element is conjugate to an element which has fully marked diagrams (this follows from Proposition \ref{su_reg}).


\bigskip

 The \emph{marked diagrams} of a normal subset $N$ of $G$ are the marked diagrams of elements in $N\cap B$ and we denote this set by $\calD(N)$. Note that $\calD(N)$ is the union of $\calD(C)$ where $C$ runs over the conjugacy classes in $N$. Notice also that, by Corollary \ref{classification}, any conjugacy class in the Jordan class $\phi(I,I',J,K)$ has a marked diagram $D$ with $\supp(D)=(J-K)\cap\Delta$. In Proposition \ref{diags_jordan_classes} we show that any conjugacy class in the Jordan class $\phi(I,I',J,K)$ has a marked diagram $D$ with $\supp(D)=(I-\Delta)\cup ((J-K)\cap\Delta)$.

 \subsection{Monoids corresponding to marked diagrams}
\label{monoids}
For an abelian monoid $M$ we use additive notation and we denote by $M\Delta$ the free monoid $M^{|\Delta|}$ with basis $\Delta$. We may view marked diagrams as elements of $M\Delta$ by identifying the basis element $\alpha$ with the marked diagram having only the node corresponding to $\alpha$ marked. We call the elements of $M\Delta$ marked diagrams.

 Denote by $\NN_{\OR}=(\{0,1\},\boxplus)$ the monoid with two elements where the operation is bitwise OR. With $\NN_{\OR}\Delta$ one can keep track of whether nodes were marked. Example:
$$
 \begin{tikzpicture}
   \draw (0,0) -- (3,0);
  \fill[white] (0,0) circle (3pt);
  \draw (0,0) circle (3pt);
  \fill (1,0) circle (3pt);
  \fill[white] (2,0) circle (3pt);
  \draw (2,0) circle (3pt);
  \fill (3,0) circle (3pt);
 \end{tikzpicture}
 \quad\boxplus\quad
  \begin{tikzpicture}
  \draw (0,0) -- (3,0);
  \fill (0,0) circle (3pt);
  \fill (1,0) circle (3pt);
  \fill[white] (2,0) circle (3pt);
  \draw (2,0) circle (3pt);
  \fill[white] (3,0) circle (3pt);
  \draw (3,0) circle (3pt);
  \end{tikzpicture}
  \quad=\quad
   \begin{tikzpicture}
   \draw (0,0) -- (3,0);
  \fill (0,0) circle (3pt);
  \fill (1,0) circle (3pt);
  \fill[white] (2,0) circle (3pt);
  \draw (2,0) circle (3pt);
  \fill (3,0) circle (3pt);
 \end{tikzpicture}.
   $$

    With $\NN\Delta$ we take into account multiplicities of marked nodes. For two diagrams $D_1=\sum_{\alpha\in\Delta}n_{\alpha}\alpha$ and $D_2=\sum_{\alpha\in\Delta}m_{\alpha}\alpha$, the sum is $D_1+D_2=\sum_{\alpha\in\Delta}(n_{\alpha}+m_{\alpha})\alpha$. Example:
 $$
 \begin{tikzpicture}
   \draw (0,0) -- (3,0);
  \fill[white] (0,0) circle (3pt);
  \draw (0,0) circle (3pt) node[label=0] {};
  \fill[white] (1,0) circle (3pt);
  \draw (1,0) circle (3pt) node[label=3] {};
  \fill[white] (2,0) circle (3pt);
  \draw (2,0) circle (3pt) node[label=0] {};
  \fill[white] (3,0) circle (3pt);
  \draw (3,0) circle (3pt) node[label=1] {};
 \end{tikzpicture}
 \quad+\quad
  \begin{tikzpicture}
   \draw (0,0) -- (3,0);
  \fill[white] (0,0) circle (3pt);
  \draw (0,0) circle (3pt) node[label=1] {};
  \fill[white] (1,0) circle (3pt);
  \draw (1,0) circle (3pt) node[label=1] {};
  \fill[white] (2,0) circle (3pt);
  \draw (2,0) circle (3pt) node[label=0] {};
  \fill[white] (3,0) circle (3pt);
  \draw (3,0) circle (3pt) node[label=0] {};
 \end{tikzpicture}
  \quad=\quad
   \begin{tikzpicture}
   \draw (0,0) -- (3,0);
  \fill[white] (0,0) circle (3pt);
  \draw (0,0) circle (3pt) node[label=1] {};
  \fill[white] (1,0) circle (3pt);
  \draw (1,0) circle (3pt) node[label=4] {};
  \fill[white] (2,0) circle (3pt);
  \draw (2,0) circle (3pt) node[label=0] {};
  \fill[white] (3,0) circle (3pt);
  \draw (3,0) circle (3pt) node[label=1] {};
 \end{tikzpicture}.
   $$
   Note that we have a partial ordering on $\NN\Delta$: if $x=\sum n_{\alpha}\Delta_\alpha$ and $y=\sum m_{\alpha}\Delta_\alpha$ then $x\geq y$ if and only if $n_\alpha-m_{\alpha}\geq 0$ for all $\alpha\in\Delta$.

   An element $\sum_{\alpha\in\Delta} n_{\alpha}\Delta_\alpha$ of $M\Delta$ is called \emph{regular} if $n_{\alpha}\neq 0$ for all $\alpha$. The fully marked diagram $\calD^{\circ}$ is regular.

   While the monoid $\NN_{\OR}\Delta$ is better suited for the description of conjugacy classes which appear in a product of conjugacy classes (as in Proposition \ref{sum_diagram_Nor}), the monoid $\NN\Delta$ is used to approximate the number of occurences of a regular conjugacy class in a product of conjugacy classes (as in Theorem \ref{normal_sets_prod_regular}). It is clear from the context which monoid is used. A connection between the two is given by the following lemma.

   \begin{lem}
  \label{monoid_morphism}
  There is a morphism of monoids $\psi:\NN\Delta\rightarrow\NN_{\OR}\Delta$ such that $\psi(D)$ is regular if and only if $D$ is regular.
  \end{lem}
\begin{proof}
  For $D=\sum_{\alpha\in\Delta}n_{\alpha}\alpha\in\NN\Delta$ define $\psi(D)$ to be $\sum_{\alpha\in\Delta}m_{\alpha}\alpha\in\NN_{\OR}\Delta$ with $m_{\alpha}=1$ if and only if $n_{\alpha}\neq 0$. The claims follow directly from the definitions.
  \end{proof}

The following lemma is needed for the proof of Proposition \ref{normal_sets_prod_regular_G}.

 \begin{lem}
  \label{sum_of_diagrams}
  Let $r=|\Delta|$ and let $D_1,\dots,D_k$ be marked diagrams such that $\sum_{i=1}^{k}D_i\geq mr\calD^{\circ}$ for some integer $m\geq 1$. There is a partition of $\{1,\dots,k\}$ into $m$ subsets $I_1,\dots,I_m$ such that $\sum_{i\in I_j}D_i\geq \calD^{\circ}$ for all $j$ with $1\leq j\leq m$.
 \end{lem}
 \begin{proof}
   We prove the lemma by induction on $m$. The case $m=1$ is trivial. Assume that $m\geq 2$. Since $\sum_{i=1}^{k}D_i\geq mr\calD^{\circ}$, there is a subset $I$ of $\{1,2,\dots,k\}$ such that $\sum_{i\in I}D_i\geq \calD^{\circ}$. Renumbering the diagrams, we may assume that $I=\{1,2,\dots,k_0\}$ where $k_0=|I|$. Let $I$ be minimal with this property, i.e. for any $j\in I$, $S_{j}:=\sum_{i\in I-\{j\}}D_i\ngeq \calD^{\circ}$. Since $\calD^{\circ}=\sum_{\alpha\in\Delta}\alpha$, for every $j\in I$ there is at least one $\alpha_j\in \Delta$ such that the coefficient of $\alpha_j$ in $S_j$ is zero. In other words, $D_j$ is the only diagram in $\{D_i\}_{i\in I}$ which has the node corresponding to $\alpha_j$ marked. This gives an injective map $I\rightarrow \Delta:j\mapsto\alpha_j$, hence $|I|\leq|\Delta|=r$. It follows that $\sum_{i\in I}D_i\leq r\calD^{\circ}$ and that $\sum_{i=k_0+1}^{k}D_i\geq (m-1)r\calD^{\circ}$. The claim follows by induction.
 \end{proof}
 
 \section{Marked diagrams and unipotent elements}
 \label{marked_diagrams_unipotent}
 
 \subsection{A map from marked diagrams to $U$}
 \label{map_to_U}
Marked diagrams attach combinatorial data to conjugacy classes by means of elements in $B$. For the reverse direction we consider the map from marked diagrams to unipotent elements given by
  \begin{equation}
    \label{ureg_map}
    u:M\Delta\rightarrow U
    ,\quad u(D):=\prod_{\alpha\in \supp(D)}u_{\alpha}(1).
  \end{equation}
It is easy to see that $\calD(u(D))=D$ for any marked diagram $D$.

\begin{lem}
  \label{regular_1}
  Let $\Delta'$ be a subset of $\Delta$. If $u\in U$ is such that $\supp(u)=\Delta'$ then there is a $T$-conjugate of $u$ in $\left(\prod_{\alpha\in\Delta'}u_{\alpha}(1)\right)[U,U]$.
\end{lem}

\begin{proof}
  Let $u=\prod_{\alpha\in\Phi^{+}}u_{\alpha}(x_{\alpha})$ with $x_{\alpha}\in F$. Consider the parabolic subgroup $P_{\Delta'}$ with Levi factor $L_{\Delta'}$ and unipotent radical $Q$. It follows that $u=\check{u}\tilde{u}$ with $\check{u}\in [L_{\Delta'},L_{\Delta'}]$ and $\tilde{u}\in Q$. For a simple root $\alpha$ we have $x_{\alpha}\neq 0$ if and only if $\alpha\in\Delta'$. Thus $\tilde u\in Q\cap [U,U]$. Since $T$ contains a maximal torus $T'$ of $[L_{\Delta'},L_{\Delta'}]$ which normalizes $Q\cap [U,U]$, it suffices to show that $\check u$ is $T'$-conjugate to an element in $\left(\prod_{\alpha\in\Delta'}u_{\alpha}(1)\right)[U',U']$ where $U'=U\cap L_{\Delta'}$.
  
  The subgroup $[L_{\Delta'},L_{\Delta'}]$ is a semisimple algebraic group, i.e. a central product of simple algebraic groups. Its maximal torus $T'$ contains a maximal torus of each simple factor, hence we may assume that $[L_{\Delta'},L_{\Delta'}]$ is a simple algebraic group. In other words we may assume that $\Delta=\Delta'$ and $T=T'$.
  
  Let $\bar u$ be the projection of $u$ on $\prod_{\alpha\in\Delta}U_{\alpha}$ and denote by $v$ the element $\prod_{\alpha\in\Delta}u_{\alpha}(1)$. By \cite[\S 3.7 Theorem 1]{Steinberg_classes}, the elements $\bar u$ and $v$ are regular. Since $\dim C_T(\bar u)=0$ the orbit $\bar u^{T}\subseteq \prod_{\alpha\in\Delta}U_{\alpha}$ has dimension $\dim T=\dim\prod_{\alpha\in\Delta}U_{\alpha}$, hence it is open in $\prod_{\alpha\in\Delta}U_{\alpha}$. Similarly for $v$. Since $\left(\prod_{\alpha\in\Delta}U_{\alpha}\right)\times\left(\prod_{\alpha\in\Phi-\Delta}U_{\alpha}\right)$ is a factorization of $U$ into algebraic sets it follows that $\bar u^{T}[U,U]=(\bar u[U,U])^{T}=(u[U,U])^{T}$ and $v^{T}[U,U]=(v[U,U])^{T}$ are open in $U$. Hence the two sets have a non-trivial intersection and the claim follows.
  \end{proof}

\begin{prop}
  \label{su_reg}
  Let $g\in B$ with $g_s\in T$ and let $\Delta_s\subseteq \supp(g_s)$. There is an element $\tilde u\in [U,U]$ such that $g$ is conjugate by an element of $B$ to $g_su(D)\tilde u$ where $D$ is the diagram with $\supp(D)=\Delta_s\cup\supp(g_u)$.
\end{prop}
\begin{proof}
  Let $D_s$ be the marked diagram with marked nodes corresponding to the simple roots in $\Delta_s$ and note that $\supp(D)=\Delta_s\cup \supp(g_u)$ is a disjoint union since $[g_s,g_u]=1$.
  
  Conjugating by an element of $T$ we may assume that the projection of $g_u$ on $U_{\alpha}$ is $u_{\alpha}(1)$ for all simple roots $\alpha\in D(g_u)$ (by Lemma \ref{regular_1}), i.e. we may assume that $g_u=u(\calD(g_u))$ modulo $[U,U]$. Thus $u(D)=g_u u(D_s)$ modulo $[U,U]$. Denote $u(D_s)$ by $\check u$ and notice that $g_su(D)=g_sg_u\check u\hat u=g\check u\hat u$ for some element $\hat u\in[U,U]$.

  By \cite[\S2.4]{Humphreys_conjugacy_classes}, the element $g\check u$ is conjugate by an element $v\in U$ to $g_su'$ with $u'\in C_U(g_s)$. The closed connected group $U$ factors as a variety: $U\cong C_U(g_s)\times Q$ where $Q$ is a closed connected subvariety of $U$ and where the isomorphism is given by the product in $G$. Hence $v=v_0v_1$ with $v_0\in C_U(g_s)$ and $v_1\in Q$. It follows that there exists $q\in Q$ such that
  $$
  g_su'=(g\check u)^{v}=(g_s^{v_0v_1})(g_u\check u)^{v}=(g_s^{v_1})(g_u\check u)^{v}
  \in g_sqg_u\check u[U,U].
  $$
Therefore
  $$
  u'\in qg_u\check u[U,U]=g_uq\check u[U,U],
  \quad\text{and thus}\quad
  g_u^{-1}u'\in q\check u[U,U].
  $$
  Since $g_u^{-1}u'\in C_U(g_s)$ the projection of $g_u^{-1}u'$ on $U_{\alpha}$ is $1$ for all $\alpha\in\supp(g_s)$. Hence the projection of $q\check u[U,U]$ on $U_{\alpha}$ is $1$ for all $\alpha\in\supp(g_s)$. In other words $q\check u\in C_U(g_s)$ modulo $[U,U]$. But $q,\check u\in Q$ hence $q\check u\in [U,U]$. It follows that $g_u^{-1}u'\in [U,U]$ and so $\tilde u=\hat u^{-1}((u')^{-1}g_u)^{v^{-1}}\in[U,U]$. Hence
  \begin{align*}
  (g_su(D)\tilde u)^v
  =
  (g\check u((u')^{-1}g_u)^{v^{-1}})^v
  =
  (g\check u)^{v}(u')^{-1}g_u
  =
  g_su'(u')^{-1}g_u
  =
  g
  .
  &\qedhere
  \end{align*}
\end{proof}

\begin{prop}
  \label{diags_jordan_classes0}
  Two conjugacy classes in the same Jordan class have the same sets of marked diagrams.
\end{prop}
\begin{proof}
  For a conjugacy class $C$ fix a diagram $D\in\calD(C)$ and let $g\in B\cap C$ be such that $D=\calD(g)$. Let $g=su$ be the factorization of $g$ with $s\in T$ and $u\in U$. By \cite[\S2.4]{Humphreys_conjugacy_classes} we may conjugate $g$ by an element of $U$ to obtain $\tilde g=s\tilde u$ with $\tilde u\in C_G(s)$. Since we conjugated by an element in $U$, $\supp(D)=\Delta_s\cup\supp(\tilde u)$ for some $\Delta_s\subseteq \supp(s)$.

  Any other conjugacy class $C'$ in the same Jordan class is represented by $h=h_sh_u$ with $h_u=\tilde u$ and $h_s\in sZ(C_G(s)^{\circ})^{\circ}$ with $C_G(h_s)^{\circ}=C_G(s)^{\circ}$ (see \S\ref{jordan_classes}). In particular $\supp(h_s)=\supp(s)$ and $\supp(h_u)=\supp(\tilde u)$. By Proposition \ref{su_reg}, $h$ is conjugate to an element $\tilde h$ with $\supp(\calD(\tilde h))=\Delta_s\cup\supp(\tilde u)$, i.e. $C'$ contains the element $\tilde h$ with $\calD(\tilde h)=D$.
\end{proof}

\begin{lem}
  \label{facts_P}
  Let $\Delta'$ be a subset of $\Delta$. Consider the standard parabolic subgroup $P_{\Delta'}=L_{\Delta'}Q$ with Levi factor $L_{\Delta'}$ and unipotent radical $Q$. Any open subset $V$ of $Q$ contains an element $u$ with $\supp(u)=\Delta-\Delta'$.
\end{lem}

\begin{proof}
  Let $\Phi_{Q}$ denote the set of roots $\Phi^{+} - \Phi_{\Delta'}$. We have $Q=\prod_{\alpha\in\Phi_{Q}}U_{\alpha}$ as subgroups of $G$. By uniqueness of expression, this is a direct product of algebraic sets. The subset $\tilde V=\prod_{\alpha\in\Phi_{Q}} (U_{\alpha} - \{ 1 \})$ is open in $Q$, hence $\tilde V\cap V\neq\emptyset$.
  \end{proof}

\begin{lem}
  \label{conj_s_by_u}
  For a semisimple element $s\in T$ and any element $u\in \prod_{\alpha\in\supp(s)}U_{\alpha}$ we have $\supp(u)=\supp([s,u])$.
\end{lem}

\begin{proof}
  An element $u$ in $\prod_{\alpha\in\supp(s)}U_{\alpha}$ is determined by the scalars $x_{\alpha}\in F$ for which $u=\prod_{\alpha\in\supp(s)}u_{\alpha}(x_{\alpha})$. We have $$[s,u]=(u^{-1})^{s}u=\left(\prod_{\alpha\in\supp(s)}u_{\alpha}(-\alpha(s)x_{\alpha}+x_{\alpha})\right)\tilde u$$ for some $\tilde u\in [U,U]$. For $\alpha\in\supp(s)$ we have $\alpha(s)\neq 1$. Thus $-\alpha(s)x_{\alpha}+x_{\alpha}\neq 0$ whenever $x_\alpha\neq 0$ and the claim follows.
\end{proof}

\begin{prop}
  \label{diags_jordan_classes}
  Any conjugacy class in the Jordan class $\phi(I,I',J,K)$ has a marked diagram $D$ with $\supp(D)=(\Delta-I)\cup ((J- K)\cap\Delta)$.
\end{prop}

\begin{proof}
  Let $g\in G$ be an element which realizes the data $(I,I',J,K)$ of the Jordan class which it represents (see \S\ref{jordan_classes}). The $[L_{J},L_{J}]$-conjugacy class of $u$ intersects the unipotent radical of $P_{K}^{[L_{J},L_{J}]}$ in an open set. Since $L_J\subseteq L_{I'}=C_G(g_s)^{\circ}$, by Lemma \ref{facts_P}, we may assume that the projections of $g_u$ on the root groups $U_{\alpha}$ with $\alpha\in J-K$ are non-trivial. We therefore have $\supp(g)=J-K$. Since $C_G(g_s)^{\circ}\subseteq L_I$, the simple roots $\Delta-I$ are in $\supp(g_s)$. Let $u=\prod_{\alpha\in\Delta-I}u_{\alpha}(1)$. By Lemma \ref{conj_s_by_u}, $\supp([g_s^{-1},u])=\supp(u)=\Delta-I$. Moreover
  $
  g^{u^{-1}}
  =(g_sg_u)^{u^{-1}}
  =g_s[g_s^{-1},u](g_u)^{u^{-1}}
  \in g_s[g_s^{-1},u]g_u[U,U]
  $. Hence, since $\supp(g_s)\cap\supp(g_u)=\emptyset$ and $\supp(u)\subseteq\supp(g_s)$, it follows that $\calD(g^{u^{-1}})$ has those nodes marked which belong to $(\Delta-I)\cup ((J- K)\cap\Delta)$.
  \end{proof}

\subsection{$U$-regular elements}
\label{U_regular}
An element of $G$ is called \emph{$U$-regular} if it is conjugate to an element of the form $su$ with $s\in T$ and $u\in U$ a regular unipotent element. A normal subset is called $U$-regular if it contains a $U$-regular element.

\begin{prop}
  \label{regular_fully_marked}
  A normal subset $N\subseteq G$ is $U$-regular if and only if $\calD(N)$ contains $\calD^{\circ}$.
\end{prop}
\begin{proof}
  By \cite[\S 3.7 Theorem 1]{Steinberg_classes} an element $u\in U$ is regular if and only if it has non-trivial projections on $U_{\alpha}$ for all $\alpha\in\Delta$. This is equivalent to $su$ having non-trivial projections on $U_{\alpha}$ for all $\alpha\in\Delta$ whenever $s$ is an element of $T$. By definition, this is equivalent to $\calD(su)=\calD^{\circ}$ for any $s\in T$.
  \end{proof}

\begin{prop}
  \label{u_reg}
  Regular conjugacy classes are $U$-regular.
\end{prop}
\begin{proof}
  Let $g$ be a representative of a regular conjugacy class such that $g_s\in T$ and $g_u\in U$. We may assume that $g$ realizes the data $(I,I',J,K)$ of the Jordan class which it represents. The element $g\in G$ is regular if and only if $g_u$ is regular in $C_G(g_s)^{\circ}$ \cite[\S3.5 Proposition 5]{Steinberg_classes}. We have $\supp(g_s)=\Delta-I'$. Since $g_u$ is regular in $C_G(g_s)$ we also have that $J=I'$ and $K=\emptyset$ \cite[\S3.7 Theorem 1]{Steinberg_classes}. Thus $\supp(g_u)=(J-K)\cap \Delta=I'\cap \Delta$. It follows from Proposition \ref{su_reg} that there is an element $\tilde u\in [U,U]$ such that $g$ is conjugate to $g_su(D)\tilde u$ where $D$ is the diagram with $\supp(D)=\supp(g_s)\cup\supp(g_u)=\Delta$.
\end{proof}

\begin{prop}
  \label{u_reg_Levi_type}
  $U$-regular conjugacy classes of proper Levi type are regular.
\end{prop}

\begin{proof}
  Let $h=su$ be a representative of a $U$-regular conjugacy class with $s\in T$ and $u$ a regular element in $U$. The element $h$ is $U$-conjugate to $g$ with $g_s=s\in T$ and $g_u\in C_U(g_s)$ \cite[\S2.4]{Humphreys_conjugacy_classes}. Conjugating in $C_G(g_s)$, we may assume that $g$ realizes the data $(I,I',J,K)$ of the Jordan class which it represents. Since $h$ is $U$-conjugate to $g$, we have $\supp(g_u)=\supp(u)\cap I$. Since the class is of proper Levi type, we have $I=I'$, hence $\supp(g_u)=\supp(u)\cap I'$. Since $g$ is $U$-regular, $\supp(g_u)=\supp(u)\cap I=I$, hence $g_u$ is regular in the Levi subgroup $C_G(g_s)^{\circ}$ \cite[\S3.7 Theorem 1]{Steinberg_classes}. The proposition follows from \cite[\S3.5 Proposition 5]{Steinberg_classes}.
  \end{proof}

\begin{prop}
  \label{U_regular_T}
  If $C_1,C_2,\dots, C_6$ are $U$-regular conjugacy classes then their product contains all regular semisimple conjugacy classes of $G$.
\end{prop}

\begin{proof}
  We may partition the set of simple roots $\Delta$ into two sets, $\Delta_1$ and $\Delta_2$, such that $\alpha$ and $\beta$ are orthogonal for any two distinct roots $\alpha,\beta\in \Delta_1$ or $\alpha,\beta\in\Delta_2$. Thus, for any two such roots $[G_{\alpha},G_{\beta}]=1$. Let $L_1$ be the standard Levi subgroup $T\prod_{\alpha\in\Delta_1}G_{\alpha}$ and let $Q_1$ be the unipotent radical of the parabolic subgroup $L_1U$.

  For $i\in\{1,\dots,6\}$, since $C_i$ is $U$-regular, it is represented by an element $s_iu_i$ with $s_i\in T$ and $u_i\in U$ regular. Moreover, $u_i=q_iv_i$ with $v_i=\prod_{\alpha\in\Delta_1}v_{i,\alpha}$ for some $q_i\in Q$ and $v_{i,\alpha}\in U_{\alpha}=G_{\alpha}\cap U$. Conjugating by $L$, and since the simple factors of $L$ commute, we may choose the elements $v_{i,\alpha}$ to be any regular unipotent element in $G_{\alpha}$. The product $C_1C_2C_3$ contains
    $$
  (s_1q_1v_1) (s_2q_2v_2) (s_3q_3v_3)=s_1s_2s_3 q_1^{s_2s_3}q_2^{(v_1^{s_2})^{-1}s_3}q_3^{(v_1^{s_2s_3})^{-1}(v_2^{s_3})^{-1}}v_1^{s_2s_3}v_2^{s_3}v_3.
  $$
  Since the set of non-trivial unipotent elements in $G_{\alpha}$ is stable by $T$-conjugation the product $v_1^{s_2s_3}v_2^{s_3}v_3$ ranges over all products of tripples of unipotent elements in $\prod_{\alpha\in\Delta_1}G_{\alpha}$. By Corollary \ref{coro_UUU}, and since $[G_{\alpha},G_{\beta}]=1$ for distinct roots $\alpha,\beta\in\Delta$, there are $v_1$, $v_2$ and $v_3$ such that $b=v_1^{s_2s_3}v_2^{s_3}v_3$ for any element $b\in \prod_{\alpha\in\Delta_1}G_{\alpha}$. Since $Q_1$ is stable under conjugation by $L$, the element $q_1^{s_2s_3}q_2^{(v_1^{s_2})^{-1}s_3}q_3^{(v_1^{s_2s_3})^{-1}(v_2^{s_3})^{-1}}$ lies in $Q_1$. In particular, for each $t\in\prod_{\alpha\in\Delta_1}(T\cap G_{\alpha})$ there is $q_t\in Q_1$ such that
  $$
  s_1s_2s_3tq_t\in C_1C_2C_3.
  $$
Similarly, for each $t\in\prod_{\alpha\in\Delta_2}(T\cap G_{\alpha})$ there is $q_t\in U$ such that
  $$
  s_4s_5s_6tq_t\in C_4C_5C_6.
  $$
  Since $T=\prod_{\alpha\in\Delta}(T\cap G_{\alpha})$, it follows that for any $t\in T$ the normal subset $C_1\cdots C_6$ contains $tq_t$ for some $q_t\in U$. In particular, if $t\in T=s_1s_2s_3s_4s_5s_6T$ is regular, then $tq_t$ is conjugate to $t$, hence $C_1\cdots C_6$ contains all regular semisimple elements of $T$.
  \end{proof}

\section{Marked diagrams and products of conjugacy classes}
\label{diagrams_and_conjugacy_classes}

\begin{lem}
  \label{T_conj}
  Let $\Delta'$ be a subset of $\Delta$. If $u_1,u_2\in U$ are such that $\supp(u_1)=\supp(u_2)=\Delta'$ then
  $$
  u_1^{T}u_2^{T}[U,U]=\left(\prod_{\alpha\in\Delta'}U_{\alpha}\right)[U,U].
  $$
\end{lem}
\begin{proof}
  The sets $u_1^{T}u_2^{T}$ and $\prod_{\alpha\in\Delta'}U_{\alpha}$ lie in the subsystem subgroup $[L_{\Delta'},L_{\Delta'}]$. As in the proof of Lemma \ref{regular_1}, it suffices to prove the claim for $\Delta'=\Delta$. Again, as in the proof of Lemma \ref{regular_1}, $u_i^{T}[U,U]$ is open in $U$. Since the product of two open sets in a connected algebraic group equals the whole group,
  \begin{align*}
  U=(u_1^{T}[U,U])(u_2^{T}[U,U])=u_1^{T}u_2^{T}[U,U].
  &\qedhere
  \end{align*}
  \end{proof}

\begin{prop}
  \label{sum_diagram_Nor}
  Let $C_1$ and $C_2$ be two conjugacy classes of $G$ represented by $g_1$ and $g_2$ respectively. If $g_1,g_2\in B$, then
  $$
  \calD(g_1)\boxplus\calD(g_2)\in\calD(C_1C_2).
  $$
\end{prop}

\begin{proof}
  For $i=1,2$ let $g_i=s_iu_i$ with $s_i\in T$ and $u_i\in U$ be the factorization of the element $g_i$ in $B=TU$. Note that $C_i$ contains $(s_iu_i)^{T}=s_i(u_i)^{T}$. Let $u_i^{1}$ be the projection of $u_i$ on $\prod_{\alpha\in\Delta}U_{\alpha}$ and $u_i^{0}\in[U,U]$ be such that $u_i=u_i^{1}u_{i}^{0}$. Then $C_i$ contains $s_i(u_i^{1})^{T}(u_i^{0})^{T}$ hence, by the commutator relations, the product $C_1C_2$ contains
$
s_1(u_1^{1})^{T}(u_1^{0})^{T}s_2(u_2^{1})^{T}(u_2^{0})^{T}
$
which lies in
$s_1s_2(u_1^{1})^{T}(u_2^{1})^{T}[U,U]$.

Let $\Delta_i=\supp(u_i^{1})$, let $v_i$ be the projection of $u_i^{1}$ on $\prod_{\alpha\in\Delta-(\Delta_1\cap\Delta_2)} U_{\alpha}$ and let $w_i$ be the projection of $u_i^{1}$ on $\prod_{\alpha\in\Delta_1\cap\Delta_2} U_{\alpha}$. We have
$$
s_1s_2(u_1^{1})^{T}(u_2^{1})^{T}[U,U]
=
s_1s_2(v_1)^{T}(w_1)^{T}(w_2)^{T}(v_2)^{T}[U,U]
$$
and, by Lemma \ref{T_conj}, this further equals
$$
s_1s_2(u_1^{1})^{T}(u_2^{1})^{T}[U,U]
=
s_1s_2(v_1)^{T}\left(\prod_{\alpha\in\Delta_1\cap\Delta_2}U_{\alpha}\right)(v_2)^{T}[U,U].
$$
Since $\supp(v_1)$, $\supp(v_2)$ and $\supp(w_1)=\supp(w_2)$ are pairwise disjoint, the above set contains an element with non-trivial projections on $U_{\alpha}$ exactly when $\alpha\in\Delta_1\cup\Delta_2$, i.e. it contains an element with diagram $\calD(g_1)\boxplus\calD(g_2)$.
\end{proof}

\begin{cor}
  \label{sum_to_regular}
  Let $N_1,\dots, N_k$ be normal subsets of $G$. If $\sum_{i=1}^{k}\calD(N_i)$ contains a regular diagram of $\NN\Delta$, then the normal subset $N_1\cdots N_k$ is $U$-regular.
\end{cor}

\begin{proof}
  Let $D_i\in\calD(N_i)$ be diagrams such that $D=\sum_{i=1}^{k}D_i$ is a regular diagram. By Lemma \ref{monoid_morphism}, $D$ is regular if and only if $\psi(D)=\boxplus_{i=1}^{k}\psi(D_i)$ is regular. Let $C_i\subseteq N_i$ be conjugacy classes such that $\psi(D_i)\in\calD(C_i)$. By Proposition \ref{sum_diagram_Nor}, $\psi(D_1)\boxplus\psi(D_2)$ is a diagram in $\calD(C_1C_2)$. By induction $\calD(C_1\cdots C_k)$ contains $\psi(D)$. By Proposition \ref{regular_fully_marked}, the normal subset $\prod_{i\in I_j}C_i$ is $U$-regular.
  \end{proof}

\begin{proof}[Proof of Proposition \ref{normal_sets_prod_regular_G}]
  As before $r=\rk(G)$. Let $D_i\in\calD(N_i)$ be diagrams such that $\sum_{i=1}^{k}D_i\geq 12r\calD^{\circ}$. Let $C_i\subseteq N_i$ be a conjugacy class such that $D_i\in\calD(C_i)$. It is enough to show that $\prod_{i=1}^{k}C_i=G$. By Lemma \ref{sum_of_diagrams}, there is a partition of $\{1,\dots,k\}$ into subsets $I_1,\dots,I_{12}$ such that $\sum_{i\in I_j}D_i\geq \calD^{\circ}$ for all $1\leq j\leq 12$. By Corollary \ref{sum_to_regular}, the normal subset $\prod_{i\in I_j}C_i$ is $U$-regular, i.e. it contains a $U$-regular conjugacy class $\tilde C_j$. By Proposition \ref{U_regular_T}, the product $\tilde C_1\cdots\tilde C_{12}$ contains $S^2$ where $S$ is the union of all regular semisimple conjugacy classes. Since $S$ is open in $G$ \cite[\S3.5 Corollary]{Steinberg_classes}, we have $S^2=G$.
\end{proof}

\begin{proof}[Proof of Proposition \ref{normal_sets_prod_order}]
  Let $D_i\in\calD(N_i)$ be diagrams such that $\{D_i:1\leq 1\leq k\}$ is of type $(r,6)$ with respect to $D_{i_1},\dots ,D_{i_r}$ and such that $\sum_{j=1}^{r}D_{i_j}\geq\calD^{\circ}$. Let $C_i\subseteq N_i$ be a conjugacy class such that $D_i\in\calD(C_i)$. Eventually after reindexing, we may assume that $N_1\cdots N_k$ contains
  $$
  \overbrace{\underbrace{C_1C_2\cdots C_6}_{D_1\leq D_2,\dots,D_6}}^{\tilde N_1}
  \overbrace{\underbrace{C_7C_8\cdots C_{12}}_{D_7\leq D_8,\dots,D_{12}}}^{\tilde N_7}
  \cdots
  \overbrace{\underbrace{C_{k}C_{k+1}\cdots C_{k+6}}_{D_k\leq D_{k+1},\dots,D_{k+6}}}^{\tilde N_k}
  \cdots
  $$
where $D_{6j+1}=D_{i_j}$. Let $\Delta_j=\supp(D_{i_j})$. As in the proof of Proposition \ref{U_regular_T}, $\tilde N_i$ contains the torus $T_j=T\cap G(\Delta_j)$. Thus, the product $\tilde N_1\cdots \tilde N_k$ contains $T_1\cdots T_k$ and since $\sum_{j=1}^{r}D_{i_j}\geq\calD^{\circ}$, the torus $T_1\cdots T_k$ is a maximal torus. Hence $\tilde N_1\cdots \tilde N_k$ contains an open set of semisimple elements.
  \end{proof}

In what follows we use the standard numering of the simple roots in $\Delta$ - see for instance \cite[p.10]{Liebeck_Seitz} or \cite[p.67]{Malle_Testerman}.

\begin{lem}
  \label{regular_torus_A}
  Let $G$ be of type $A_r$. If $\Delta=\{\alpha_1,\dots,\alpha_r\}$ are the simple roots ordered as above then the torus $\prod_{i=1}^{\lceil r/2\rceil}\alpha_{2i-1}^{\vee}(k^{\times})$ is regular.
\end{lem}

\begin{proof}
  Let $\tilde T=\prod_{i=1}^{\lceil r/2\rceil}\alpha_{2i-1}^{\vee}(k^{\times})$ and let $\tilde\Delta=\{\alpha_1,\alpha_3,\dots\alpha_{2\lceil r/2\rceil-1}\}$. Since $\tilde T\subseteq T$, by \cite[II \S4.1]{Springer_Steinberg} we have $C_{G}(\tilde T)^{\circ}=\langle T,U_{\alpha}:\alpha(t)=1\text{ for all }t\in \tilde T\rangle$. For a root $\alpha\in\Phi$ we have $\alpha(t)=t^{\langle\alpha,\beta\rangle}$ for any $t\in k^{\times}$ as in \eqref{cochar_action}. Hence, $\alpha(t)=1$ for all $t\in \tilde T$ if and only if $\langle \alpha,\beta\rangle=0$ for all $\beta\in\tilde\Delta$, i.e. if and only if $\alpha$ is orthogonal to $\tilde\Delta$. One checks, for example with \cite[\S2.10]{Humphreys_Ref_Coxeter}, that no root $\alpha$ is orthogonal to $\tilde\Delta$. Hence $C_G(\tilde T)^{\circ}=T$, i.e. $\tilde T$ is a regular torus.
\end{proof}

\begin{lem}
  \label{regular_torus_elem}
  Let $\tilde T\subseteq T$ be a regular torus of $G$. For any $\tilde z\in T$, the set $\tilde z\tilde T$ contains an open subset of regular elements.
\end{lem}
\begin{proof}
  Fix $\alpha\in\Phi$. 
  Since $\tilde T$ and $\tilde z\tilde T$ lie in $T$, the root subgroup $U_{\alpha}$ is stable by conjugation under the elements of these closed subsets. Since $\tilde T$ is a regular torus, it has an open orbit on $U_{\alpha}$. Thus, restricting the conjugation map to $\tilde z\tilde T$ one also obtains an open orbit on $U_{\alpha}$. Hence $\dim(C_G(U_{\alpha})\cap \tilde z\tilde T)\lneq \dim(z\tilde T)$, i.e. $z\tilde T\setminus C_G(U_{\alpha})$ is an open subset of $\tilde z\tilde T$. Consider $V=\tilde z\tilde T\setminus \cup_{\alpha\in\Phi}C_G(U_{\alpha})=\cap_{\alpha\in\Phi}\tilde z\tilde T\setminus C_G(U_{\alpha})$. Since $|\Phi|$ is finite, $V$ is an open subset of $\tilde z\tilde T$. Since the ground field is algebraically closed, $V$ is non-empty. For any $t\in V$ and any root subgroup $U_{\alpha}$ we have $[t,U_{\alpha}]\neq 1$, thus, by \cite[II \S4.1]{Springer_Steinberg}, all elements of $V$ are regular.
  \end{proof}

\begin{lem}
  \label{thmA_typeA}
  Let $G$ be of type $A_r$. Let $\tilde \Delta=\{\alpha_{2i-1}:1\leq i\leq\lceil r/2\rceil\}$ and let $\tilde T=\prod_{i=1}^{\lceil r/2\rceil}\alpha_{2i-1}^{\vee}(k^{\times})$. Let $N_1,\dots,N_k$ be normal subsets of $G$ and let $D_{i} \in \calD(N_{i})$ for all $1 \leq i \leq k$.
  If $\sum_{i=1}^{k} D_{i}\geq 2\calD^{\circ}$, then $\prod_{i=1}^{k}N_i$ contains an open subset of $\tilde T\tilde z$ for some $\tilde z\in T$. In particular, Theorem \ref{normal_sets_prod_regular} holds for $G$ of type $A_r$. 
\end{lem}
\begin{proof}
  Let $D_i\in\calD(N_i)$ be diagrams such that $\sum_{i=1}^{k}D_i\geq 2\calD^{\circ}$. Let $C_i\subseteq N_i$ be a conjugacy class such that $D_i\in\calD(C_i)$. It is enough to show that $\prod_{i=1}^{k}C_i$ contains a regular semisimple element. Let $g_i\in C_i$ be such that $\calD(g_i)=D_i$.
  
  Notice that $\tilde T$ is a maximal torus of the derived group $[L_{\tilde\Delta},L_{\tilde\Delta}]$ of the Levi subgroup $L_{\tilde\Delta}$ of type $A_1^{\lceil r/2\rceil}$. Thus, the representatives factor as $g_i=l_iq_i$ for some $l_i\in L$ and $q_i\in Q$. Moreover $l_i=(\prod_{j=1}^{\lceil r/2\rceil}l_{i,j})z_i$ for some $l_{i,j}\in G_{\alpha_{j}}\cap B$ and $z_i\in Z(L)$. For $x\in L$ we have 
  $$
  \prod_{i=1}^{k}g_i^{x}
  =
  \prod_{i=1}^{k}(\prod_{j=1}^{\lceil r/2\rceil}l_{i,j}^{x})(z_iq_i)^{x}
  =
  \left[\prod_{i=1}^{k}\prod_{j=1}^{\lceil r/2\rceil}l_{i,j}^{x}\right]\tilde z\tilde q
  =
  \left[\prod_{j=1}^{\lceil r/2\rceil}\prod_{i=1}^{k}l_{i,j}^{x}\right]\tilde z\tilde q
  $$
  for some $\tilde z\in Z(L)$ and $\tilde q\in Q$. Since $\sum_{i=1}^{k}D_i\geq 2\calD^{\circ}$, for each $j$ there are at least $2$ elements in $\{l_{i,j}:1\leq i\leq \lceil r/2\rceil\}$ which are not central in $G_{\alpha_j}$. Then $\prod_{i=1}^{k}l_{i,j}^{L}$ contains an open subset of $G_{\alpha_{2i-1}}\subseteq [L,L]$ by \cite[Theorem 2]{Gordeev_Saxl_1}. Thus $\{\prod_{j=1}^{\lceil r/2\rceil}\prod_{i=1}^{k}l_{i,j}^{x}:x\in L\}$ contains an open subset of $\tilde T$. 
  Notice that, since $z_i\in Z(L)$, $\tilde z$ depends only on the choice of representatives $g_i$ and not on $x\in L$. 
  By Lemma \ref{regular_torus_elem}, there is an open subsets of regular (semisimple) elements $V\subseteq \tilde T\tilde z$. Thus, for any $\tilde t\tilde z\in V$, $\tilde t\tilde z\tilde q$ is conjugate to $\tilde t\tilde z$. Hence, the normal subset $\prod_{i=1}^{k}N_i$ contains $V$.
  Since the ground field is algebraically closed, $V$ is non-empty.
  \end{proof}

\begin{proof}[Proof of Theorem \ref{normal_sets_prod_regular}]
  Let $D_i\in\calD(N_i)$ be diagrams such that $\sum_{i=1}^{k}D_i\geq 16\calD^{\circ}$. Let $C_i\subseteq N_i$ be a conjugacy class such that $D_i\in\calD(C_i)$. It is enough to show that $\prod_{i=1}^{k}C_i$ contains a regular semisimple element. Let $g_i\in C_i$ be such that $\calD(g_i)=D_i$. Notice that $k\geq 16$ with equality if and only if all $D_i$ are regular. If $k=16$ the claim follows from Proposition \ref{U_regular_T}. Let $k\geq 17$. If $\rk(G)\leq 8$ then $\prod_{i=1}^{k}C_i$ contains an open subset of $G$ by \cite[Theorem 2]{Gordeev_Saxl_1} and the claim follows.

  Suppose $r=\rk(G)> 8$ and let $\Delta_1=\{\alpha_1,\dots,\alpha_{r-4}\}$ and $\Delta_2=\{\alpha_{r-2},\alpha_{r-1},\alpha_{r}\}$. Consider the parabolic subgroup $P_{\Delta_1\cup\Delta_2}$ with Levi factor $L$ and unipotent radical $Q$. Then $g_i=l_iq_i$ for some $l_i\in L$ and $q_i\in Q$. Moreover, $L$ is a connected reductive group with a simple factor $H_1$ of type $A_{r-4}$ corresponding to $\Delta_1$ and a simple factor $H_2$ corresponding to $\Delta_2$. Thus
  $$
  g_i=g_{i,1}g_{i,2}z_iq_i
  $$
  for some $g_{i,1}\in H_1$, $g_{i,2}\in H_2$ and $z_i\in Z(L)$. For $x\in L$ we have 
  $$
  \prod_{i=1}^{k}g_i^{x}
  =
  \prod_{i=1}^{k}g_{i,1}^{x}g_{i,2}^{x}(z_iq_i)^{x}
  =
  \left[\prod_{i=1}^{k}g_{i,1}^{x}g_{i,2}^{x}\right]\tilde z\tilde q
  =
  \left[\prod_{i=1}^{k}g_{i,1}^{x}\right]\left[\prod_{i=1}^{k}g_{i,2}^{x}\right]\tilde z\tilde q
  $$
  for some $\tilde q\in Q$ and $\tilde z\in Z(L)$. Notice that, since $z_i\in Z(L)$, $\tilde z$ depends only on the choice of representatives $g_i$ and not on $x\in L$.

  Let $\Delta_1'=\{\alpha_{2i-1}:1\leq i\leq\lceil r/2-2\rceil\}$, let $\tilde T=\prod_{i=1}^{\lceil r/2-2\rceil}\alpha_{2i-1}^{\vee}(k^{\times})$ and let $T_{\Delta_2}=T\cap H_2$. Since $\sum_{i=1}^{k}D_i\geq 16\calD^{\circ}$, the elements $g_{i,1}$ and $g_{i,2}$ are non-central in $H_1$ and $H_2$ respectively. Moreover $k\geq 17$. Since $H_1$ and $H_2$ commute, by Lemma \ref{regular_torus_A} and \cite[Theorem 2]{Gordeev_Saxl_1}, $\prod_{i=1}^{k}C_i$ contains $\tilde TT_{\Delta_2}$ up to multiplication on the right by elements in $\tilde zQ$. 

  We claim that the torus $\tilde TT_{\Delta_2}$ is regular. If this is the case then, by Lemma \ref{regular_torus_elem}, there is a regular element $\tilde t\tilde z\in \tilde TT_{\Delta_2}\tilde z$, thus, all elements in $\tilde t\tilde z Q$ are regular semisimple and the proof is finished.
  

  For the proof of the claim, let $\Delta^{\perp}=\{\alpha\in\Phi:\alpha\perp \beta\text{ for all }\beta\in \Delta_1'\cup \Delta_2\}$. We have $C_G(\tilde TT_{\Delta_2})^{\circ}=\langle T, U_{\alpha}:\alpha\in\Delta^{\perp}\rangle$ and need to show that $\Delta^{\perp}=\emptyset$.
  For this we use the explicit construction of root systems as described in \cite[\S2.10]{Humphreys_Ref_Coxeter}. If $G$ is of type $A$, let $\alpha=\varepsilon_i-\varepsilon_j$ ($1\leq i\neq j\leq r+1$) then either $\beta=\varepsilon_{i-1}-\varepsilon_{i}$ or $\beta=\varepsilon_{i}-\varepsilon_{i+1}$ is an element of $\Delta_1'\cup\Delta_2$, hence $\Delta^{\perp}=\emptyset$. If $G$ is of type $B$ and $\alpha$ is the short root $\pm\varepsilon_i$ ($1\leq i\leq r$) then either $\beta=\varepsilon_{i-1}-\varepsilon_{i}$ or $\beta=\varepsilon_{i}-\varepsilon_{i+1}$ or $\beta=\varepsilon_{i}$ is an element of $\Delta_1'\cup\Delta_2$, hence $\alpha\not\perp \Delta_1'\cup \Delta_2$. If $\alpha$ is the long root $\pm\varepsilon_i\pm\varepsilon_j$ ($1\leq i< j\leq r$), again $\beta=\varepsilon_{i-1}-\varepsilon_{i}$ or $\beta=\varepsilon_{i}-\varepsilon_{i+1}$ or $\beta=\varepsilon_{i}$ is an element of $\Delta_1'\cup\Delta_2$, hence $\alpha\not\perp \Delta_1'\cup \Delta_2$. If $G$ is of type $C$ or $D$, the argument is similar and the claim follows.
  \end{proof}

\end{document}